\documentclass{amsart}[12pt]
\usepackage{graphicx}
\usepackage{amsmath}
\usepackage{amsthm}
\usepackage{amsfonts}
\usepackage{amssymb, amscd}

\usepackage{color}

\usepackage[all]{xy}

\newtheorem{thm}{Theorem}[section]
\newtheorem{cor}[thm]{Corollary}
\newtheorem{lem}[thm]{Lemma}
\newtheorem{prop}[thm]{Proposition}
\newtheorem{example}[thm]{Example}
\theoremstyle{definition}
\newtheorem{defn}[thm]{Definition}
\theoremstyle{remark}
\newtheorem{rem}[thm]{Remark}
\numberwithin{equation}{section}
\numberwithin{equation}{section}

\newcommand{\K}{\mathbb K}

\newcommand{\A}{\mathcal{A}}

\newcommand{\sll}{\mathfrak{sl}_2(\mathbb{K})}

\input{epsf.sty}

\def\a{{\mathfrak a}}

\newcommand{\C}{\mathbb{C}}
\newcommand{\Z}{\mathbb{Z}}

\newcommand{\LX} {\mathcal{L}}

\begin{document}

\title[Cohomology and Deformations of Hom-algebras]{ Cohomology and Deformations of Hom-algebras}

\author{F. Ammar, Z. Ejbehi and A. Makhlouf }%

\address{Faouzi Ammar and Zeyneb Ejbehi, Universit{\'{e}} de Sfax, Facult{\'{e}} des Sciences, B.P.
1171, 3000 Sfax, Tunisia} \email{Faouzi.Ammar@rnn.fss.tn, ejbehizeyneb@yahoo.fr}
\address{Abdenacer Makhlouf  Universit\'{e} de Haute Alsace,  Laboratoire de Math\'{e}matiques, Informatique et Applications,
4, rue des Fr\`{e}res Lumi\`{e}re F-68093 Mulhouse, France}%
\email{Abdenacer.Makhlouf@uha.fr}

\maketitle


\begin{abstract}
The purpose of this paper is to define  cohomology structures on Hom-associative algebras and Hom-Lie algebras. The first and second coboundary maps were  introduced by Makhlouf and Silvestrov in the study of one-parameter formal  deformations theory. Among  the relevant formulas for a generalization of Hochschild cohomology for Hom-associative algebras and  a Chevalley-Eilenberg
cohomology for  Hom-Lie algebras, we define  Gerstenhaber bracket on the space of multilinear mappings of Hom-associative algebras and  Nijenhuis-Richardson bracket on the space
of multilinear mappings of Hom-Lie algebras. Also we enhance the deformations theory of this  Hom-algebras by studying the obstructions.
\end{abstract}

 \section*{Introduction}
Hom-Type algebras have been recently investigated by many authors. The main feature of these
algebras is that the identities defining the structures are twisted
by homomorphisms.  Such algebras appeared in the ninetieth in
examples of $q$-deformations of the Witt and the Virasoro algebras. Motivated by these
examples and their generalization,   Hartwig, Larsson and Silvestrov
introduced and studied  in \cite{{Harwig Silv}} the classes of
quasi-Lie, quasi-Hom-Lie and Hom-Lie algebras. In the class of
Hom-Lie algebras skew-symmetry is untwisted, whereas the Jacobi
identity is twisted by a homomorphism and contains three terms as in
Lie algebras, reducing to ordinary Lie algebras when the twisting
linear map is the identity map.

The  Hom-associative algebras play the role of associative algebras
in the Hom-Lie setting. They were introduced by Makhlouf and
Silvestrov  in \cite{Makhl Silv Hom}, where it is shown that
 the commutator bracket of a Hom-associative algebra
gives rise to a Hom-Lie algebra. Given a Hom-Lie algebra,  a universal enveloping Hom-associative
 algebra was constructed by Yau in \cite{Yau:EnvLieAlg}. The Hom-Lie superalgebras have been studied by Ammar and Makhlouf in \cite{F.Ammar}. In a similar way  several other algebraic structures have been investigated.

  The one-parameter formal deformations of Hom-associative algebras and Hom-Lie algebras were studied  by Makhlouf and Silvestrov in \cite{Makhl Silv Not}. The authors introduced  the first
  and second cohomology spaces of Hom-associative algebras and Hom-Lie algebras, which fits with the deformation theory.

  The purpose of this paper is to enhance the cohomology study initiated in \cite{Makhl Silv Not}. We consider multiplicative Hom-associative algebras and Hom-Lie algebras. Among other the following main results   are obtained:

 (1) We define a Gerstenhaber bracket on the space of multilinear mappings of Hom-associative algebras and the Richardson-Nijenhuis-bracket on the space
of multilinear mappings of Hom-Lie algebras.

(2) We provide a Hochschild cohomology of Hom-associative algebras and
a   Chevalley-Eilenberg cohomology of  Hom-Lie algebras, extending in one hand these cohomologies to Hom-algebras situation and in the other hand generalizing the first and second coboundary maps introduced in \cite{Makhl Silv Not}.

 The  paper is organized as  follows.
  In the first Section  we summarize the definitions of Hom-algebras of different  type (see \cite{F.Ammar},\cite{Harwig Silv},\cite{Makhl Silv Not},\cite{Makhl Silv Hom}) and  present some  preliminary results on graded algebras  (see \cite{Gerst def}, \cite{Lecomte}). In  Section 2 we define a cohomology structure of Hom-associative algebras
 and a cohomology structure of  Hom-Lie algebras. Section 3 is dedicated to study  $C_{\alpha}(\A,\A)$, the set of multilinear mappings $\varphi$ satisfying
 $\alpha (\varphi(x_0,...,x_{n-1}))=\varphi(\alpha(x_0),...,\alpha(x_{n-1}))
 \; \forall\;x_0,...,x_{n-1} \in \A$, where  $\alpha : \A\rightarrow \A$ is a morphism. It is endowed with a structure of   graded Lie algebra $(C_{\alpha}(\A,\A),[.,.]_\alpha^{\Delta})$, where  $[.,.]_\alpha^{\Delta}$ is
   the Gerstenhaber bracket. Henceforth, we provide  a cohomology differential operator
  $D^{\alpha}_{\mu}=[\mu,.]^{\Delta}_{\alpha}$
 on $C_\alpha (\A,\A)$ where $(\A,\mu,\alpha)$
 is a Hom-associative algebra such that $\alpha(\mu(x,y))=\mu(\alpha(x),\alpha(y)).$ We denote by $H^{*}_{D} (\A,\A)$ the corresponding  cohomology spaces
  and we show that $H^{*}_{D} (\A,\A)=H^{*+1}_{Hom} (\A,\A)$ where $H^{*+1}_{Hom} (\A,\A)$ is the space of Hochschild cohomology of  Hom-associative algebras.
Also we study the graded algebra $(\tilde{C}_{\alpha}(\LX,\LX),[.,.]_\alpha^{\wedge})$
    of multilinear $\varphi$ mapping satisfying
  $\alpha (\varphi(x_0,...,x_{n-1}))=\varphi(\alpha(x_0),...,\alpha(x_{n-1}))$
   for all $x_0,...,x_{n-1} \in \LX$ where $(\LX,[.,.],\alpha)$ is Hom-Lie algebra such that
  $\alpha([x,y])=[\alpha(x),\alpha(y)]$ and $[.,.]_\alpha^{\wedge}$ is the
  Nijenhuis-Richardson bracket. Similarly, we provide a cohomology differential operator of
  $D^{\alpha}_{[.,.]}=[[.,.],.]^{\Delta}_{\alpha}$. We denote $H^{*+1}_{HL}
  (\LX,\LX)$ the corresponding space of cohomology and we show that $H^{*}_{D} (\LX,\LX)=H^{*+1}_{HL} (\LX,\LX)$ where
   $H^{*+1}_{HL} (\LX,\LX)$ is the space of  Chevalley-Eilenberg cohomology of the Hom-Lie algebra.
In the last Section, we recall and enhance  the one-parameter formal deformation  theory of Hom-associative algebras and Hom-Lie algebras introduced  in \cite{Makhl Silv Not}, we study in particular the obstructions involving the third cohomology groups.

 Throughout this paper $\K$ denotes an algebraically closed field of characteristic 0.

\section { Preliminaries}
In this Section  we summarize the definitions of Hom-type algebras and provide some examples   (see \cite{F.Ammar},\cite{Harwig Silv} , \cite{Makhl Silv Not},\cite{Makhl Silv Hom}) and  present some  preliminary results on graded algebras  (see \cite{Gerst def}, \cite{Lecomte}).
\subsection{Hom-algebras}
We mean by Hom-algebra a triple $(A,\mu,\alpha)$ consisting of a $\K$-vector space $A$, a bilinear map $\mu : A \times A \longrightarrow A $   and a linear   map $ \alpha :A \rightarrow A$. The main feature of  Hom-algebra structures is that the classical identities are   twisted by the linear map. We summarize in the following the definitions of Hom-associative algebra, Hom-Lie algebras and Hom-Poisson algebras.
\begin{defn}
A Hom-associative algebra is a triple $(\A,\mu,\alpha)$ consisting of a $\K$-vector space $\A$, a bilinear map $\mu : \A \times \A \rightarrow \A $  and  a linear   map $ \alpha :\A \rightarrow \A$  satisfying $$\mu(\alpha(x),\mu(y,z))=\mu(\mu(x,y),\alpha(z))\; \;\hbox{for all}\; x,y,z \in \A\;\qquad \;\hbox{(Hom-associativity)}$$
We refer by $\A$ to the Hom-associative algebra when there is no ambiguity.
\end{defn}
\begin{rem}
When $\alpha$ is the identity map, we recover the classical
associative algebra.
\end{rem}
\begin{example}\label{ex}
Let $\A$ be  a $2$-dimensional  vector space over $\K$, generated by $\{e_1,e_2\}$,
 $\mu : \A
 \times \A \rightarrow \A $ be a multiplication  defined by
\begin{itemize}
\item $\mu(e_1,e_1)=e_1$
 \item $\mu(e_i,e_j)=e_2$  if $(i,j)\neq(1,1)$
 \end{itemize}
 and  $\alpha : V\rightarrow V$ be a linear map defined by $\alpha(e_1)=\lambda e_1+\gamma e_2,\;  \alpha(e_2)=(\lambda+\gamma) e_2 \;\hbox{where}\; \lambda,\gamma \in \K^* .$

 Then $(V,\mu,\alpha)$ is a Hom-associative algebra.
\end{example}

\begin{defn}
A Hom-Lie algebra   is a triple $(\LX,[.,.],\alpha)$ consisting by  a $\K$-vector space $\LX$, a bilinear map $[.,.] : \LX\times \LX \rightarrow \LX $  and  a linear  map $ \alpha : \LX \rightarrow \LX$  satisfying $$ [x,y]=-[y,x]\;\hbox{for all}\; x,y \in \LX\;\qquad \;\hbox{(skew-symmetry)}, $$
$$\hbox{and}\; \circlearrowleft_{x,y,z}\big[\alpha(x),[y,z]\big]=0\;\;\hbox{for all}\; \;x,y,z \in \LX\;\qquad \;\hbox{(Hom-Jacobi identity)}$$
where $ \circlearrowleft_{x,y,z}$ denotes summation over the cyclic permutation on $x, y, z$.

We refer by $\LX$ to the Hom-Lie algebra when there is no ambiguity.
\end{defn}
\begin{rem}
We recover the classical Lie algebra when $\alpha = id$.
\end{rem}
\begin{example}[\cite{Makhl Silv Hom}]
$(\sll (2,\C),[\cdot,\cdot],\alpha)$ is a $3$-dimensional  Hom-Lie algebra generated by  $$ H=\left(
\begin{array}{cc}
1 & 0 \\
0 & -1 \\
\end{array}
\right),
\\E=\left(
  \begin{array}{cc}
  0 & 0 \\
    1 & 0 \\
     \end{array}
      \right)
      ,F=\left(
      \begin{array}{cc}
       0 & 1 \\
          0 & 0 \\
           \end{array}
           \right),
$$
with $[A,B]=AB-BA$ and where the twist maps are given with respect to the basis by the matrices
$$
\mathcal{M}_\alpha=\left(
\begin{array}{ccc}
a & c & d \\
2d & b & e \\
2c & f & b \\
\end{array}
\right)\; \hbox{where}\; a,b,c,d,e,f \in \C,$$
\end{example}
Let $\left(\A,\mu,\alpha \right) $ and
$\left( \A^{\prime },\mu ^{\prime
},\alpha^{\prime }\right) $ (resp. $\left(\LX,[.,.],\alpha \right) $ and
$\left( \LX^{\prime },[.,.]^{\prime
},\alpha^{\prime }\right) $) be two Hom-associative  (resp. Hom-Lie) algebras. A linear map
$\phi\ :\A\rightarrow \A^{\prime }$ (resp. $\phi\ :\LX\rightarrow \LX^{\prime }$) is
a morphism of Hom-associative (resp. Hom-Lie)  algebras if%
$$
\mu ^{\prime }\circ (\phi\otimes \phi)=\phi\circ \mu \quad (\text{resp.  } [.,.]^{\prime }\circ (\phi\otimes \phi)=\phi\circ [.,.]) \qquad \text{ and }\qquad \phi\circ \alpha=\alpha^{\prime }\circ \phi.
$$

Now, we define Hom-Poisson algebras introduced in \cite{Makhl Silv Not}. This structure
 emerged naturally in   deformation theory. It is shown that a one-parameter formal deformation of commutative Hom-associative algebra leads to a Hom-Poisson algebra.

\begin{defn}
A Hom-Poisson algebra is a quadruple $(A,\mu, \{\cdot,\cdot\}, \alpha)$ consisting of
 a vector space $A$, bilinear maps $\mu: A\times A \rightarrow A$ and
  $\{\cdot, \cdot\}: A\times A \rightarrow A$, and
 a linear map $\alpha: A \rightarrow A$
 satisfying
 \begin{enumerate}
\item $(A,\mu, \alpha)$ is a commutative Hom-associative algebra,
\item $(A, \{\cdot,\cdot\}, \alpha)$ is a Hom-Lie algebra,
\item
for all $x, y, z$ in $A$,
\begin{equation}\label{CompatibiltyPoisson}
\{\alpha (x) , \mu (y,z)\}=\mu (\alpha (y),
\{x,z\})+ \mu (\alpha (z), \{x,y\}).
\end{equation}
\end{enumerate}
\end{defn}

\begin{example}\label{example1HomPoisson}
Let $\{x_1,x_2,x_3\}$  be a basis of a $3$-dimensional vector space
$A$ over $\K$. The following multiplication $\mu$, skew-symmetric
bracket and linear map $\alpha$ on $A$ define a Hom-Poisson algebra
over $\K^3${\rm :}
$$
\begin{array}{ll}
\begin{array}{lll}
 \mu ( x_1,x_1)&=&  x_1, \ \\
\mu ( x_1,x_2)&=& \mu ( x_2,x_1)=x_3,\\
 \end{array}
 & \quad
 \begin{array}{lll}
\{ x_1,x_2 \}&=& a x_2+ b x_3, \ \\
\{ x_1, x_3 \}&=& c x_2+ d x_3, \ \\
  \end{array}
\end{array}
$$

$$  \alpha (x_1)= \lambda_1 x_2+\lambda_2 x_3 , \quad
 \alpha (x_2) =\lambda_3 x_2+\lambda_4 x_3  , \quad
   \alpha (x_3)=\lambda_5 x_2+\lambda_6 x_3
$$
where $a,b,c,d,\lambda_1,
\lambda_2,\lambda_3,\lambda_4,\lambda_5,\lambda_6 $ are parameters
in $\K$.
\end{example}

\subsection{Graded Lie algebras}
In the following we recall the definition of $\Z$-graded Lie algebra and elements of Gerstenhaber algebra which endow the set of classical cochains, see \cite{Gerst def,Lecomte}.
\begin{defn}
A pair $(A,[.,.])$ is a $\Z$-graded Lie algebra if
\begin{enumerate}
\item $A$  is a graded algebra, i.e. it is a direct summation  of vector subspaces, $A=\bigoplus_{n\in\Z}A^n$, such that $[A^n,A^m]\subset A^{n+m}$,
\item the bracket $[.,.]$ in A is graded skew-symmetric, i.e.
\begin{equation}\label{gradedSkewSym}
[x,y]=-(-1)^{pq} [y,x]\;\;\hbox{for}\; x \in A^p, y\in A^q ,
\end{equation}
\item it satisfies the so called  graded Jacobi identity : \begin{equation}\label{gradedJacobi}\circlearrowleft_{x,y,z}(-1)^{pq}\big[x,[y,z]\big]=0,\;\;\hbox{for}\; x \in A^p, y\in A^r, z\in A^q.
    \end{equation}
\end{enumerate}
\end{defn}
\begin{rem}\label{rem}
It is easy to check that if $\pi\in A^1$ is such that $[\pi,\pi]=0$ then the map $\delta_\pi^p : A^p\rightarrow A^{p+1}$ defined by $ \delta_\pi^p (x)=[\pi,x]$ is a coboundary map, i.e. $\delta_\pi^{p+1}\circ\delta_\pi^p=0.$ Indeed, from \ref{gradedJacobi} we have $$[[\pi,\pi],x]=2 [\pi,[\pi,x]]=2 \delta_\pi^{p+1}\big(\delta_\pi^p(x)\big).$$
\end{rem}

\
Let $A$ be a $\K$-vector space and  $M^k(A,A)$ be the space of $(k+1)$-linear maps $K : A^{\times k}\rightarrow A$ and set $M(A,A)=\bigoplus_{k\in\Z}M^k(A,A)$.
In \cite{Gerst def,Lecomte}, the graded Lie algebra $(M(A,A),[.,.]^\Delta)$ is described for each vector space $A$ with the property  that $(A,\mu)$ is an associative algebra  if and only if $\mu \in M^1(A,A)$ and $[\mu,\mu]^\Delta=0$. This algebra is defined as follows:
\\For $K_i\in M^{k_i}$ and $x_j\in A$ we define $j_{K_1}K_2\in M^{k_1+k_2}(A)$ by $$j_{K_1}K_2(x_0,...,x_{k_1+k_2})=\sum_{i=0}^{k_2} (-1)^{k_{1}i}K_2(x_0,...,K_1(x_i,...,x_{k_{1}+i}),...,x_{k_{1}+k_2}).$$
In particular, if $k_1=k_2=1$ we have $j_{K_1}K_2(x_0,x_1,x_{2})=K_2(K_1(x_0,x_{1}),x_{2})-K_2(x_0,K_1(x_{1},x_{2}))$ which is denoted sometimes by $K_2\circ K_1$.

The graded Lie bracket on $M(A,A)$ is then given by $$[K_1,K_2]^\Delta=j_{K_1}K_2-(-1)^{k_1k_2}j_{K_2}K_1.$$
The graded Jacobi identity is a consequence of the  formula $$j_{[K_1,K_2]^\Delta}=[j_{K_1},j_{K_2}], \quad \text{ where} \quad [.,.]\ \  \text{ is the graded commutator in} \ \  End(M(A,A)).$$
Also in \cite{Gerst def,Lecomte}, the graded Lie algebra $(\lambda(M(A,A)),[.,.]^\wedge)$ is described for each vector space $A$ with the property  that $(A,[.,.])$ is a Lie algebra if and only if $[.,.] \in M^1(A,A)$ and $\big[[.,.],[.,.]\big]^\wedge=0$. This algebra is as follows: \\For the alternator operator $\lambda : M(A,A)\rightarrow M(A,A)$ they defined  $(\lambda(M(A,A))$ as the space of alternating cochains and similarly one  defines
$$i_{K_1}(K_2):= \frac{(k_1+k_2+1)!}{(k_1+1)!(k_2+1)!}\lambda(j_{K_1}(K2)).$$
The graded Lie bracket of $\lambda(M(A,A))$ is then given by
$$[K_1,K_2]^\wedge=\frac{(k_1+k_2+1)!}{(k_1+1)!(k_2+1)!}\lambda([K_1,K_2]^\Delta)=i_{K_1}K_2-(-1)^{k_1k_2}i_{K_2}K_1$$
if $K_1\in M^{k_1}(A,A)$ and  $K_2\in M^{k_2}(A,A)$ then $i_{K_1}K_2 \in \lambda(M^{k_1+k_2}(A,A)).$
 The graded Jacobi identity is a consequence of the following formula
$$\lambda\big(j_{\lambda(K_1)}\lambda(K_2)\big)=\lambda(j_{K_1}K_2).$$

\section {Cohomologies of Hom-associative algebras and Hom-Lie algebras}
 The first and the second cohomology groups of  Hom-associative algebras and Hom-Lie algebras were introduced in \cite{Makhl Silv Not}. The aim  of  this section is to
  construct cochain complexes that define cohomologies of these Hom-algebras with the assumption that they are multiplicative.

\subsection{Cohomology of multiplicative  Hom-associative algebras}
The purpose of  this section is to
  construct cochain complex $C_{Hom}^*(\A,\A)$ of a multiplicative Hom-associative algebra $\A$ with coefficients in $\A$ that defines a cohomology $H_{Hom}^*(\A,\A)$.

Let $(\A,\mu,\alpha)$ be a Hom-associative algebra, for $n\geq 1$ we define  a $\K-$vector space $C_{Hom}^n(\A,\A)$ of $n$-cochains as follows : \\a cochain $\varphi\in C_{Hom}^n(\A,\A)$ is an $n$-linear map $\varphi : \A^n \rightarrow \A \; \hbox{satisfying}$ $$\alpha \circ \varphi(x_0,...,x_{n-1})=\varphi\big(\alpha (x_0),\alpha(x_1),...,\alpha(x_{n-1})\big) \; \hbox{for all}\; x_0,x_1,...,x_{n-1} \in \A.$$
\begin{defn}
We call, for $n\geq1$,  $n$-coboundary operator of the Hom-associative algebra $(\A,\mu,\alpha)$ the linear map $\delta_{Hom}^{n} : C_{Hom}^n(\A,\A) \rightarrow C_{Hom}^{n+1}(\A,\A)$ defined by \begin{align}\label{HomCohomo}\delta_{Hom}^n \varphi (x_0,x_1,...,x_n)& =\mu \big(\alpha^{n-1}(x_0),\varphi(x_1,x_2,...,x_n)\big)\\ \ & +\sum_{k=1}^n (-1)^k \varphi\big(\alpha(x_0),\alpha(x_1),...,\alpha(x_{k-2}),\mu(x_{k-1},x_k),\alpha(x_{k+1}),...,\alpha(x_n)\big)\nonumber \\ \ &+(-1)^{n+1}\mu\big(\varphi(x_0,...,x_{n-1}),\alpha^{n-1}(x_n)\big).\nonumber
\end{align}
\end{defn}
\begin{lem}
Let $D_i : C_{Hom}^n(\A,\A)\rightarrow
C_{Hom}^{n+1}(\A,\A)$ be the linear operators defined for $\varphi\in C_{Hom}^n(\A,\A)$ and $x_0,x_1,...,x_n\in \A$  by
\begin{align*}
& D_0\varphi(x_0,x_1,...,x_n)=-\mu(\alpha^{n-1}(x_0),\varphi(x_1,...,x_n))+ \varphi(\mu(x_0,x_1),\alpha(x_2),...,\alpha(x_n)),\\
& D_i\varphi(x_0,x_1,...,x_n)=\varphi(\alpha(x_0),...,\mu(x_i,x_{i+1}),...,\alpha(x_n)) \quad \hbox{for}\quad  1\leq i\leq n-2,\\
&  D_{n-1}\varphi(x_0,...,x_n)=\varphi(\alpha (x_0),...,,...\alpha(x_{n-2}),\mu(x_{n-1},x_n))
-\mu(\varphi(x_0,...,x_{n-1}),\alpha^{n-1}(x_n)),
\\ &  D_i\varphi=0 \quad  \hbox{for}\quad  i\geq n.
\end{align*} Then
$$D_iD_j=D_jD_{i-1}\quad  0\leq j<i\leq n,\quad  \text{ and }\quad  \delta_{Hom}^n=\sum_{i=0}^n(-1)^{i+1}D_i.$$
\end{lem}
\begin{prop}\label{prop f0}
Let $(\A,\mu,\alpha)$ be a Hom-associative algebra and $\delta_{Hom}^n : C_{Hom}^n(\A,\A) \rightarrow C_{Hom}^{n+1}(\A,\A)$ be the operator  defined in \eqref{HomCohomo} then \begin{equation}\delta_{Hom}^{n+1} \circ \delta_{Hom}^n =0   \;\;\;\hbox{for} \; n \geq1. \end{equation}
\end{prop}
\begin{proof}
Indeed
\begin{align*}
\delta_{Hom}^{n+1} \circ \delta_{Hom}^n & =\sum_{0\leq i,j \leq n}(-1)^{i+j}D_iD_j=\sum_{0\leq j<i \leq n}(-1)^{i+j}D_iD_j+\sum_{0\leq i\leq j \leq n}(-1)^{i+j}D_iD_j
\\  \ & = \sum_{0\leq j<i \leq n}(-1)^{i+j}D_jD_{i-1}+\sum_{0\leq i\leq j \leq n}(-1)^{i+j}D_iD_j\\
\ & =\sum_{0\leq j\leq k \leq n}(-1)^{k+j+1}D_jD_k+\sum_{0\leq i\leq j \leq n}(-1)^{i+j}D_iD_j\\
\ & =0.
\end{align*}
\end{proof}
\begin{rem}
A proof of the previous proposition could also be obtained as a consequence of Propositions (\ref{P1}) and (\ref{P2}).
\end{rem}
\begin{defn}
The space of $n-$cocycles is defined by $$Z_{Hom}^n(\A,\A)=\{\varphi\in C_{Hom}^n(\A,\A):\ \delta_{Hom}^n\varphi=0\},$$and the space of $n-$couboundary is defined by $$B_{Hom}^n(\A,\A)=\{\psi=\delta_{Hom}^{n-1}\varphi :\ \varphi \in C^{n-1}(\A,\A) \}.$$
\end{defn}
\begin{lem}
$B_{Hom}^n(\A,\A) \subset Z_{Hom}^n(\A,\A)$.
\end{lem}
\begin{defn}
We call the $n^{th}$ cohomology group of the Hom-associative algebra $\A$ the quotient
$$H_{Hom}^n(\A,\A)=\frac{Z_{Hom}^n(\A,\A)}{B_{Hom}^n(\A,\A)}.$$
\end{defn}
\begin{rem}
The cohomology class of an element $\varphi \in C_{Hom}^n(\A,\A)$ is given by the set of elements  $\psi$ such that $\psi=\varphi+\delta^{n-1}f$ where $f$ is a $(n-1)$-cochain.
\end{rem}
\begin{example}
We consider the example (\ref{ex})   of   Hom-associative algebras with $\lambda+\gamma=0$ i.e. the matrix of the twist map  $\alpha$  is $\lambda\cdot \left(
 \begin{array}{cc}
 1 & 0 \\
 -1 & 0 \\
 \end{array}
   \right)
$.  We obtain with respect to the same basis \begin{itemize}
\item $Z_{Hom}^2(\A,\A)=\{ \psi  /\psi(e_1,e_1)=a e_1+b e_2, \; \;\psi(e_i,e_j)=c e_2\quad \hbox{if }\; (i,j)\;\neq(1,1)\}$
\item $B_{Hom}^2(\A,\A)=\{\delta f/\delta f(e_1,e_1)=a e_1 + b e_2,\delta f(e_i,e_j)=(a+b)e_2\quad \hbox{if }\;(i,j)\;\neq(1,1)\}$
     \end{itemize}
then
 \begin{itemize}
        \item $H_{Hom}^2(\A,\A)=\{\psi/\psi(e_1,e_1)=ae_1+be_2,\psi(e_i,e_j)=ce_2\quad \hbox{if }\;(i,j)\neq(1,1)\; c\neq a+b\}$
      \item $ H_{Hom}^3(\A,\A)=0$
      \end{itemize}
\end{example}

\subsection{Cohomology of multiplicative  Hom-Lie algebras}
The purpose of  this section is to
  construct cochain complex $C_{HL}^{*}(\LX,\LX)$ of a multiplicative Hom-Lie algebra $\LX$ with coefficients in $\LX$ that defines a cohomology $H_{HL}^*(\LX,\LX)$.

Let $(\LX,[.,.],\alpha)$ be a Hom-Lie algebra. We define,  for $n\geq 1$, a $\K$-vector space $C_{HL}^n(\LX,\LX)$ of $n$-linear alternating cochains as follows:\\a cochain $\varphi\in C_{HL}^{n}(\LX,\LX)$ is an $n$-linear alternating map $\varphi : \LX^n \rightarrow \LX \; \hbox{satisfying}$ $$\alpha \circ \varphi(x_0,...,x_{n-1})=\varphi\big(\alpha (x_0),\alpha(x_1),...,\alpha(x_{n-1})\big) \; \hbox{for all}\; x_0,x_1,...,x_{n-1} \in \LX.$$
\begin{defn}
 We call, for $n\geq1$,  $n$-coboundary operator of the Hom-Lie algebra $(\LX,[.,.],\alpha)$ the linear map $\delta_{HL}^{n} : C_{HL}^n(\LX,\LX) \rightarrow C_{HL}^{n+1}(\LX,\LX)$ defined by
\begin{align}\label{HLcohomo}
\delta_{HL}^n\varphi(x_0,x_1,...,x_n)& =\sum_{k=0}^n(-1)^k \big[\alpha^{n-1}(x_k),\varphi(x_0,...,\widehat{x_k},...,x_n)\big]\\
 \ & +\sum_{0\leq i<j\leq n}\varphi([x_i,x_j],\alpha(x_0),...,\widehat{x_i},...,\widehat{x_j},...,\alpha(x_n))\nonumber
\end{align}
where $\widehat{x_k}$ designed that $x_k$ is omitted.
\end{defn}
\begin{defn}
The space of $n-$cocycles is defined by $$Z_{HL}^n(\LX,\LX)=\{\varphi \in \tilde{C}^n(\LX,\LX):\ \delta_{HL}^n\varphi=0\},$$and the space of $n-$couboundaries is defined by $$B_{HL}^n(\LX,\LX)=\{\psi=\delta_{HL}^{n-1}\varphi :\  \varphi \in \tilde{C}^{n-1}(\LX,\LX) \}.$$
\end{defn}
\begin{prop}\label{prop f1}
Let $(\LX,[.,.],\alpha)$ be a Hom-Lie algebra and $\delta_{HL}^{n} :C_{HL}^n(\LX,\LX) \rightarrow C_{HL}^{n+1}(\LX,\LX)$ be the operator defined \eqref{HLcohomo}.    Then
\begin{equation}\delta_{HL}^{n+1} \circ \delta_{HL}^n =0 \quad \;\hbox{for} \;\quad  n \geq1.\end{equation}
\end{prop}
\begin{proof}
The proof can be  obtained by a long straightforward calculation or as a consequence of propositions (\ref{P3}) and (\ref{P4}).
\end{proof}
\begin{rem}
One has $B_{HL}^n(\LX,\LX) \subset Z_{HL}^n(\LX,\LX)$.
\end{rem}
\begin{defn}We call the $n^{th}$ cohomology group of the Hom-Lie algebra $\LX$ the quotient
$$H_{HL}^n(\LX,\LX)=\frac{Z_{HL}^n(\LX,\LX)}{B_{HL}^n(\LX,\LX)}.$$
\end{defn}

\section{Gerstenhaber algebra and Nijenhuis-Richardson algebra}
We define in this section two graded Lie algebras on the space of multilinear (resp. alternating multilinear) mappings which are multiplicative with respect to a linear map $\alpha$.

\subsection{The algebra $C_\alpha(\A,\A)$}\label{sec3}
We provide in this section a variation of Gerstenhaber algebra supplying the set of all multiplicative  multilinear maps on a given vector space.  Let $A$ be a vector space and  $\alpha : A\rightarrow A$ be a linear map.
We denote by $C_\alpha^n(A,A)$ the space of all $(n+1)$-linear maps $\varphi :
A^{ \times(n+1)}\rightarrow A$ satisfying
\begin{equation}\label{MultipliMultiLinearMap}\alpha
(\varphi(x_0,...,x_{n}))=\varphi\big(\alpha(x_0),...,\alpha(x_{n})\big)\;\hbox{for
all}\;x_0,...,x_{n} \in A
\end{equation}
We set
$$C_\alpha(A,A)=\bigoplus_{n\geq-1}C_\alpha^n(A,A).$$
 If $\varphi \in C_\alpha^a(A,A)
\;\hbox{and}\;\psi\in C_\alpha^b(A,A)$ where $a\geq0, b\geq0$ then we define $j_{\varphi}^{\alpha}(\psi)
\in C_\alpha^{a+b+1}(A,A)$ by
$$j_{\varphi}^{\alpha}(\psi)(x_0,...,x_{a+b})=\sum_{k=0}^b(-1)^{ak}  \psi\big(\alpha^a(x_0),...,\alpha^a(x_{k-1}),\varphi(x_k,...,x_{k+a}),\alpha^a(x_{a+k+1}),...,\alpha^a(x_{a+b})\big).$$
 and $$[\psi,\varphi]_\alpha^\Delta=j_\psi^\alpha(\varphi)-(-1)^{ab}j_\varphi^\alpha(\psi)$$

The bracket $[.,.]_\alpha^\Delta$ is called Gerstenhaber bracket.

\begin{rem}
If $a=b=1$ we have $j_{\psi}\varphi(x_0,x_1,x_{2})=\varphi(\psi(x_0,x_{1}),\alpha(x_{2}))-\varphi(\alpha (x_0),\psi(x_{1},x_{2}))$ which is denoted in \cite{Makhl Silv Not} by $\varphi\circ_\alpha \psi$. The particular case, where $\varphi=\psi$ corresponds to the Hom-associator.
\end{rem}

\begin{lem}We have
$j_{[\varphi,\psi]_\alpha^\Delta}=[j_\varphi^\alpha,j_\psi^\alpha]\;\hbox{for all }\;\varphi,\psi \in C_\alpha(A,A),$ where $[.,.]$ is the graded commutator on $End(C_\alpha(A,A)).$
\end{lem}
\begin{proof}
Let $\varphi\in C_\alpha^a(A,A),\psi\in C_\alpha^b(A,A),\xi \in C_\alpha^c(A,A)$ $$[j_\varphi^\alpha,j_\psi^\alpha](\xi)(x_0,....,x_{a+b+c})=\big(j_\varphi^\alpha(j_\psi^\alpha \xi)-(-1)^{ab}j_\psi^\alpha (j_\varphi^\alpha\xi)\big)(x_0,....,x_{a+b+c})$$ $$=S_1-(-1)^{ab} S_2.$$
 where $$S_1=j_\varphi^\alpha\big(j_\psi^\alpha( \xi)\big)(x_0,...,x_{a+b+c})\;\hbox{and}\; S_2=j_\psi^\alpha (j_\varphi^\alpha\xi)\big)(x_0,....,x_{a+b+c}).$$
  We have $$S_1=\sum_{k=0}^{b+c}(-1)^{ak}j_\psi^\alpha (\xi)\big(\alpha^a(x_0),..,\alpha^a(x_{k-1}),\varphi(x_k,..,x_{k+a}),\alpha^a(x_{a+k+1}),..,\alpha^a(x_{a+b+c})\big)$$ $$=A+B+C$$
 where
 \begin{align*}
 A=\sum_{k=b+1}^{b+c}\sum_{i=0}^{k-(b+1)}(-1)^{ak+bi} \xi\big(\alpha^{a+b}(x_0),...,\alpha^{a+b}(x_{i-1}),\psi(\alpha^a(x_i),...,\alpha^a(x_{i+b}))
,\alpha^{a+b}(x_{i+b+1}),...,\alpha^{a+b}(x_{k-1}),\\ \alpha^b(\varphi(x_k,...,x_{k+a})),\alpha^{a+b}(x_{a+k+1}),...,\alpha^{a+b}(x_{a+b+c})\big)
\end{align*}
\begin{align*}
B=\sum_{k=0}^{c}\sum_{i=k-b}^{k}(-1)^{ak+bi} \xi\big(\alpha^{a+b}(x_0),...,\alpha^{a+b}(x_{i-1}),\psi(\alpha^a(x_i),...,\alpha^{a}(x_{k-1}),
\varphi(x_k,...,x_{k+a}),\alpha^a(x_{k+a+1}),\\ ...,\alpha^a(x_{a+b+i})),\alpha^{a+b}(x_{a+b+i+1}),...,\alpha^{a+b}(x_{a+b+c})\big)
\end{align*}
\begin{align*}
C=\sum_{k=0}^{c-1}\sum_{i=a+k+1}^{a+c}(-1)^{ak+b(i-a)} \xi\big(\alpha^{a+b}(x_0),...,\alpha^{a+b}(x_{k-1}),\alpha^b(\varphi(x_k,...,x_{k+a})) ,\alpha^{a+b}(x_{a+k+1}), ...,\\ \psi(\alpha^a(x_i),...,\alpha^a(x_{i+b})),
...,\alpha^{a+b}(x_{a+b+c})\big)
\end{align*}
 We obtain $S_2$ if we permute $\varphi$ and $\psi $.
 $$S_2=D+E+F$$ where
 \begin{align*}
 D=\sum_{k=a+1}^{a+c}\sum_{i=0}^{k-(a+1)}(-1)^{ak+bi}  \xi\big(\alpha^{a+b}(x_0),...,\alpha^{a+b}(x_{i-1}),\varphi(\alpha^b(x_i),...,\alpha^b(x_{i+a}))
,\alpha^{a+b}(x_{i+a+1}),...,\alpha^{a+b}(x_{k-1}) ,\\ \alpha^a(\psi(x_k,...,x_{k+b})),\alpha^{a+b}(x_{a+k+1}),...,\alpha^{a+b}(x_{a+b+c})\big)
\end{align*}
 \begin{align*}
  E=\sum_{k=0}^{c}\sum_{i=k-b}^{k}(-1)^{ak+bi}\xi\big(\alpha^{a+b}(x_0),...,\alpha^{a+b}(x_{i-1}),\varphi(\alpha^b(x_i),
  ...,\alpha^{b}(x_{k-1}),\psi(x_k,...,x_{k+b}),\alpha^b(x_{k+b+1}), ...,\\ \alpha^b(x_{a+b+i})),\alpha^{a+b}(x_{a+b+i+1}),...,\alpha^{a+b}(x_{a+b+c})\big)
  \end{align*}
 \begin{align*}
 F=\sum_{k=0}^{c-1}\sum_{i=b+k+1}^{b+c}(-1)^{bk+a(i-b)} \xi\big(\alpha^{a+b}(x_0),...,\alpha^{a+b}(x_{k-1}),\alpha^a(\psi(x_k,...,x_{k+b})),\alpha^{a+b}(x_{b+k+1}), ...,\\ \varphi(\alpha^b(x_i),...,\alpha^b(x_{i+a})),
...,\alpha^{a+b}(x_{a+b+c})\big)
  \end{align*}
Since $$\alpha \circ \varphi(x_0,...,x_{a})=\varphi\big(\alpha (x_0),\alpha(x_1),...,\alpha(x_{a})\big),$$ then $$\alpha^b(\varphi(x_0,...,x_{a}))=\varphi\big(\alpha^b (x_0),\alpha^b(x_1),...,\alpha^b(x_{a})\big).$$
\\So, $A-(-1)^{ab}F=0,\; C-(-1)^{ab}D=0$ and
 \begin{align*} [j_\varphi^\alpha,j_\psi^\alpha](\xi)&= B-(-1)^{ab}E \\  \ & =\sum_{k=0}^{c}\sum_{i=k-b}^{k}(-1)^{ak+bi} \xi\big(\alpha^{a+b}(x_0),...,\alpha^{a+b}(x_{i-1}),\psi(\alpha^a(x_i),...,\alpha^{a}(x_{k-1}),\varphi(x_k,...,
 x_{k+a}),\\  \ &  \alpha^a(x_{k+a+1}), ..., \alpha^a(x_{a+b+i})),\alpha^{a+b}(x_{a+b+i+1}),...,\alpha^{a+b}(x_{a+b+c})\big)\\  \ & -(-1)^{ab} \sum_{k=0}^{c}\sum_{i=k-a}^{k}(-1)^{ai+bk} \xi\big(\alpha^{a+b}(x_0) ,..., \alpha^{a+b}(x_{i-1}), \varphi(\alpha^b(x_i),...,\alpha^{b}(x_{k-1}),
 \psi(x_k,...,x_{k+b}),\\  \ & \alpha^b(x_{k+b+1}),..., \alpha^b(x_{a+b+i})),\alpha^{a+b}(x_{a+b+i+1}),...,\alpha^{a+b}(x_{a+b+c})\big)
 \\ \ &=j_{[\varphi,\psi]_\alpha^\Delta}(\xi).
 \end{align*}
\end{proof}
\begin{thm}
The pair $\big(C_\alpha(V,V),[.,.]_\alpha^\Delta \big)$
 is a graded Lie algebra.
\end{thm}
\begin{proof}
The proof is based on the previous Lemma.
 Let  $\varphi\in C_\alpha^a(V,V),\psi\in C_\alpha^b(V,V),\phi \in C_\alpha^c(V,V).$
\begin{enumerate}
\item \textbf{Skew-symmetry}
\begin{align*}[\varphi,\psi]_\alpha^\Delta &=j_\varphi^\alpha\psi-(-1)^{ab}j_\psi^\alpha\varphi \\ \ & =(-1)^{ab+1}\big(j_{\psi}^\alpha \varphi-(-1)^{ab}j_\varphi^\alpha\psi\big) \\ \ & =(-1)^{ab+1}[\psi,\varphi]_\alpha^\Delta.
\end{align*}
\item \textbf{Graded Hom-Jacobi identity}
\begin{align*}
\circlearrowleft_{\varphi,\psi,\phi}(-1)^{ac}\big[\varphi,[\psi,\phi]_\alpha^\Delta\big]_\alpha^\Delta & =
(-1)^{ac}j_\varphi^\alpha[\psi,\phi]_\alpha^\Delta-(-1)^{ab}j_{[\psi,\phi]_\alpha^\Delta}\varphi\\ \ &
+(-1)^{ba} j_\psi^\alpha[\phi,\varphi]_\alpha^\Delta-(-1)^{bc}j_{[\phi,\varphi]_\alpha^\Delta}\psi
\\ \ &  +(-1)^{cb} j_\phi^\alpha[\varphi,\psi]_\alpha^\Delta-(-1)^{ca}j_{[\varphi,\psi]_\alpha^\Delta}\phi\\ \ & =(-1)^{ac}j_\varphi^\alpha(j_\psi^\alpha \phi-(-1)^{cb}j_\phi^\alpha\psi)-(-1)^{ab}j_{[\psi,\phi]_\alpha^\Delta}\varphi\\ \ & +(-1)^{ba}j_\psi^\alpha(j_\phi^\alpha\varphi-(-1)^{ac}j_\varphi^\alpha\phi)-(-1)^{bc}j_{[\phi,\varphi]_\alpha^\Delta}\psi \\ \ & +(-1)^{cb}j_\phi^\alpha (j_\varphi^\alpha\psi-(-1)^{ab}j_\psi^\alpha\varphi)-(-1)^{ca}j_{[\varphi,\psi]_\alpha^\Delta}\phi.
\end{align*}
                Organizing  these terms leads to
\begin{align*}
\circlearrowleft_{\varphi,\psi,\phi}(-1)^{ac}\big[\varphi,[\psi,\phi]_\alpha^\Delta\big]_\alpha^\Delta & =(-1)^{ba}\big(j_\psi^\alpha(j_\phi^\alpha\varphi)-(-1)^{cb}j_\phi^\alpha (j_\psi^\alpha\varphi)-j_{[\psi,\phi]_\alpha^\Delta}\varphi\big)\\ \ &
 +(-1)^{cb}\big(j_\phi^\alpha (j_\varphi^\alpha\psi)-(-1)^{ac}j_\varphi^\alpha(j_\phi^\alpha\psi) -j_{[\phi,\varphi]_\alpha^\Delta}\psi\big) \\ \ &
  +(-1)^{ac}\big(j_\varphi^\alpha (j_\psi^\alpha \phi)-(-1)^{ab}j_\psi^\alpha (j_\varphi^\alpha\phi)-j_{[\varphi,\psi]_\alpha^\Delta}\phi\big)\\ \ & =(-1)^{ba}\big([j_\psi^\alpha,j_\phi^\alpha]-j_{[\psi,\phi]_\alpha^\Delta}\big)\varphi\\ \ & +(-1)^{cb}\big([j_\phi^\alpha,j_\varphi^\alpha]-j_{[\phi,\varphi]_\alpha^\Delta}\big)\psi\\ \ & +(-1)^{ac}\big([j_\varphi^\alpha,j_\psi^\alpha]-j_{[\varphi,\psi]_\alpha^\Delta}\big)\phi.
  \end{align*}
  Using the previous lemma we get
  $$\circlearrowleft_{\varphi,\psi,\phi}(-1)^{ac}\big[\varphi,[\psi,\phi]_\alpha^\Delta\big]_\alpha^\Delta=0.$$
  \end{enumerate}
  \end{proof}
\begin{prop}\label{P1}
Let $(\A,\mu,\alpha)$ be a Hom-associative algebra. Let  $ D_\mu^\alpha : C_\alpha(\A,\A) \rightarrow
C_\alpha(\A,\A)$ be a linear map defined by $$D_\mu^\alpha \phi=[\mu,\phi]_\alpha^\Delta \;\quad
\hbox{for all}\; \phi\in C_\alpha(\A,\A).$$
Then  $ D_{\mu}^\alpha $ is a differential operator, and for $\phi \in C_\alpha^{n-1}(\A,\A)$ we have
$D_\mu^\alpha \phi=-\delta_{Hom}^n \phi.$
\end{prop}
\begin{proof}
Let  $\phi \in C_\alpha^{n-1}(\A,\A)$ and $x_0,...,x_n\in\A$,
\begin{align*}
D_\mu^\alpha\phi(x_0,...,x_n)&=[\mu,\phi]^\Delta(x_0,...,x_n)=\big(j_\mu^\alpha(\phi)-(-1)^{n-1}j_\phi^\alpha(\mu)\big)(x_0,...,x_n)\\
& =\sum_{k=0}^{n-1}(-1)^{k}
\phi\big(\alpha(x_0),...,\alpha(x_{k-1}),\mu(x_k,x_{k+1}),\alpha(x_{k+2}),...,\alpha(x_n)\big)\\
& -(-1)^{n-1}\mu(\phi(x_0,...,x_{n-1}),\alpha^{n-1}(x_n))-(-1)^{n-1}
(-1)^{n-1}\mu\big(\alpha^{n-1}(x_0),\phi(x_1,...,x_{n})\big)\\
& =-\big(\mu\big(\alpha^{n-1}(x_0),\phi(x_1,...,x_{n})\big)+\sum_{k=1}^{n}(-1)^{k}
\phi\big(\alpha(x_0),...,\alpha(x_{k-2}),\mu(x_{k-1},x_{k}),\\ & \alpha(x_{k+1}),...,\alpha(x_n)\big)+(-1)^{n+1}\mu(\phi(x_0,...,x_{n-1}),\alpha^{n-1}(x_n))\big)\\ & =-
\delta_{Hom}^n(\phi)
\end{align*}
\end{proof}
Let $(\A,\mu,\alpha)$ be a Hom-algebra, it is easy to see that
$[\mu,\mu]_\alpha^\Delta=0$ if and only if $(\A,\mu,\alpha)$ is a
Hom-associative algebra.\\Indeed, let $x,y,z \in \A$
\begin{align*}
[\mu,\mu]_\alpha^\Delta\big(x,y,z\big)& =\big(j_\mu^\alpha\mu-(-1)^1j_\mu^\alpha\mu \big)(x,y,z)=2j_\mu^\alpha\mu(x,y,z)\\ \  & =2\big(\mu(\mu(x,y),\alpha(z))-\mu(\alpha(x),\mu(y,z))\big).
\end{align*}
Henceforth, if we use the remark (\ref{rem}) we obtain the following proposition:
\begin{prop}\label{P2}
The differential operator $ D_\mu^\alpha : C_\alpha(\A,\A)\rightarrow C_\alpha(\A,\A)$ satisfies
$(D_\mu^\alpha)^2=0.$
\end{prop}
\begin{rem}
The proof of the fundamental proposition (\ref{prop f0}) is a direct
consequence of the propositions (\ref{P1}) and (\ref{P2}).
\end{rem}
We denote the corresponding space of $(n+1)-$cocycles for the coboundary operator $ D_\mu^\alpha$ by $$Z_{D}^n(\A,\A)=\{\varphi\in C_\alpha^n(\A,\A): \
D_\mu^\alpha\varphi=0\},$$
and the space of $(n+1)-$couboundaries by
$$B_{D}^n(\A,\A)=\{D_\mu^\alpha\varphi : \  \varphi\in C_\alpha^{n-1}(\A,\A)\}.$$
Hence the corresponding cohomology is given by
$$H_{D}^n(\A,\A)=\frac{Z_{D}^n(\A,\A)}{B_{D}^n(\A,\A)}.$$
\begin{rem}
The relationship with the cohomology $H_{Hom}^{*}(\A,\A)$ introduced  above is
  $$B_{D}^n(\A,\A)=B_{Hom}^{n+1}(\A,\A),\quad  Z_{D}^n(\A,\A)=Z_{Hom}^{n+1}(\A,\A) \quad \hbox{and } \; H_{D}^n(\A,\A)=H_{Hom}^{n+1}(\A,\A).$$
\end{rem}
\subsection{The algebra $\tilde{C}_\alpha(\LX,\LX)$ }\label{secCTilde}
Let $A$ be a vector space and $\alpha : A \rightarrow A$ be a linear
map. We denote by $\tilde{C}_\alpha^n(A,A)$ the space of all
$(n+1)$-alternating linear maps $\varphi : A^{\times
(n+1)}\rightarrow A$  satisfying for
all $x_0,...,x_{n} \in A$
$$\alpha
(\varphi(x_0,...,x_{n}))=\varphi(\alpha(x_0),...,\alpha(x_{n})),$$ and set\\
$$\tilde{C}_\alpha(A,A)=\bigoplus_{n\geq-1}\tilde{C}_\alpha^n(A,A).$$

We define the alternator $\lambda : C_\alpha(A,A)\rightarrow C_\alpha(A,A)$ by $$(\lambda\varphi)(x_0,...,x_a)=\frac{1}{(a+1)!}\sum_{\sigma\in \mathcal{S}_{a+1}}\varepsilon(\sigma)\varphi(x_{\sigma(0)},...,x_{\sigma(a)})\;\hbox{for}\; \varphi \in C_\alpha^a(A,A).$$
where $\mathcal{S}_{a+1}$ is the permutation group and $\varepsilon(\sigma)$ is the signature of the permutation $\sigma$.
\begin{rem}
The set $\tilde{C}_\alpha(A,A)$ may be viewed as images by $\lambda$ of the elements of $C_\alpha(A,A).$
\end{rem}
\begin{lem}
The alternator $\lambda : C_\alpha(A,A)\rightarrow C_\alpha(A,A)$ satisfies $\lambda^2=\lambda,$
and we have $$\lambda\big(j_{\lambda(\varphi)}^\alpha\lambda(\psi)\big)=\lambda(j_\varphi^\alpha\psi)\; \hbox{for each}\; \varphi,\psi \in C_\alpha(A,A).$$
\end{lem}
\begin{proof}
The proof is similar to the classical case $(\alpha=id)$.
\end{proof}
We define an  operator and a  bracket for $\varphi \in C_\alpha^a(A,A)$ and  $\psi\in C_\alpha^b(A,A)$ by $$i_\varphi^\alpha(\psi):= \frac{(a+b+1)!}{(a+1)!(b+1)!}\lambda(j_\varphi^\alpha\psi),$$
$$[\varphi,\psi]_\alpha^\wedge:=\frac{(a+b+1)!}{(a+1)!(b+1)!}\lambda([\varphi,\psi]_\alpha^\Delta)=
i_\varphi^\alpha(\psi)-(-1)^{ab}i_\psi^\alpha(\varphi).$$
Thus $i_{\varphi}^{\alpha}(\psi) \in \tilde{C}_\alpha^{a+b+1}.$

The bracket $[.,.]_\alpha^\wedge$ is called Nijenhuis-Richardson bracket.
\begin{thm}
The pair $(C_\alpha(A,A),[.,.]_\alpha^\wedge)$ is a graded Lie algebra.

In particular, $(\tilde{C}_\alpha(A,A),[.,.]_\alpha^\wedge)$ is a graded Lie algebra.
\end{thm}
\begin{proof}
Let  $\varphi \in C_\alpha^a(A,A), \; \psi\in C_\alpha^b(A,A)\; \hbox{and}\; \phi \in C_\alpha^c(A,A)$ $$\circlearrowleft_{\varphi,\psi,\phi}(-1)^{ac}\big[\varphi,[\psi,\phi]_\alpha^\wedge\big]_\alpha^\wedge
=\frac{(a+b+c+1)!}{(a+1)!(b+1)!(c+1)!}\circlearrowleft_{\varphi,\psi,\phi}
\lambda(\big[\varphi,\lambda([\psi,\phi]_\alpha^\Delta)\big]_\alpha^\Delta).$$
Notice that, $$\lambda\big([\varphi,\psi]_\alpha^\Delta\big)= \lambda\big([\lambda(\varphi),\lambda(\psi)]_\alpha^\Delta\big)\;\hbox{and}\;\lambda^2=\lambda.$$
Then,
\begin{align*}
\circlearrowleft_{\varphi,\psi,\phi}(-1)^{ac}\big[\varphi,[\psi,\phi]_\alpha^\wedge\big]_\alpha^\wedge &=\frac{(a+b+c+1)!}{(a+1)!(b+1)!(c+1)!}\lambda \big(\circlearrowleft_{\varphi,\psi,\phi}\big[\varphi,[\psi,\phi]_\alpha^\Delta\big]_\alpha^\Delta\big)\\ \ & =0.
\end{align*}
\end{proof}
The following lemma is a generalization to twisted case of a result in \cite{Lecomte}.
\begin{lem}\label{lem}
Let $\varphi \in C_\alpha^a(A,A), \; \psi\in C_\alpha^b(A,A)$. Then $$i_{\varphi}^{\alpha}(\psi)(x_0,...,x_{b+a})=\frac{1}{b!(a+1)!}\sum_{\sigma\in S_{a+b+1}}\varepsilon(\sigma)\psi\big(\varphi(x_{\sigma(0)},...,x_{\sigma(a)}),\alpha^a(x_{\sigma(a+1)}), ...,\alpha^a(x_{\sigma(a+b)})\big)$$
\end{lem}
\begin{prop}\label{P3}
Let $(\LX,[.,.],\alpha)$ be a Hom-Lie algebra,
 the linear map $ D_{[.,.]}^\alpha : \tilde{C}_\alpha(\LX,\LX)
\rightarrow \tilde{C}_\alpha(\LX,\LX)$ is defined  by
$$D_{[.,.]}^\alpha(\phi)=\big[[.,.],\phi\big]_\alpha^\wedge \;
\hbox{for all}\; \phi\in \tilde{C}_\alpha(\LX,\LX).$$
 Therefore $ D_{[.,.]}^\alpha $ is a differential operator, and for $\phi \in \tilde{C}_\alpha^{n-1}(\LX,\LX)$ we have
$D_{[.,.]}^\alpha(\phi)=\delta_{HL}^n(\phi).$
\end{prop}
\begin{proof}
The proof is obtained using  Lemma (\ref{lem}) and  straightforward calculation.
\end{proof}
A Hom-algebra $(\LX,[.,.],\alpha)$
is Hom-Lie algebra  if and only if
$\big[[.,.],[.,.]\big]_\alpha^\wedge=0.$

Indeed, let $x,y,z \in \LX$
\begin{align*}
\big[[.,.],[.,.]\big]_\alpha^\wedge\big(x,y,z\big)&=\big(i_{[.,.]}^\alpha[.,.]-(-1)^1i_{[.,.]}^\alpha[.,.]\big)(x,y,z)
\\ \ & =2i_{[.,.]}^\alpha[.,.](x,y,z)\\ \ & = 2\big( \circlearrowleft_{x,y,z}[[x,y],\alpha(z)]\big).
\end{align*}
Thus, using the remark (\ref{rem}) we have the following proposition:
\begin{prop}\label{P4}
The differential operator $ D_{[.,.]}^\alpha : C(\LX,\LX)\rightarrow C_\alpha(\LX,\LX)$ satisfies
$(D_{[.,.]}^\alpha)^2=0.$
\end{prop}
\begin{rem}
The proof of the fundamental Proposition (\ref{prop f1}) is a direct
consequence of the propositions (\ref{P3}) and (\ref{P4}).
\end{rem}

 We denote the corresponding space of  $(n+1)$-cocycles for the coboundary operator $D_{[.,.]}^\alpha $ by
 $$\tilde{Z}_D^n(\LX,\LX)=\{\varphi\in C_\alpha^n(\LX,\LX):\ D_{[.,.]}^\alpha\varphi=0\},$$
 the space of $(n+1)$-coboundaries by  $$\tilde{B}_D^n(\LX,\LX)=\{D_{[.,.]}^\alpha\varphi:\  \varphi\in \tilde{C}^{n-1}(\LX,\LX)\}$$
  and the corresponding cohomology group by $$\tilde{H}_D^n(\LX,\LX)=\frac{\tilde{Z}_D^n(\LX,\LX)}{\tilde{B}_D^n(\LX,\LX)}.$$
\begin{rem}The relationship with the cohomology $H_{HL}^{*}(\LX,\LX)$ introduced  above is
$$\tilde{B}_D^n(\LX,\LX)=B_{HL}^{n+1}(\LX,\LX),\quad  \tilde{Z}_D^n(\LX,\LX)=Z_{HL}^{n+1}(\LX,\LX)\; \quad \hbox{and} \; \tilde{H}_D^n(\LX,\LX)=H_{HL}^{n+1}(\LX,\LX).$$
\end{rem}

\section{One-parameter formal deformations}
The one-parameter formal deformation of Hom-associative algebras and Hom-Lie algebras were introduced in \cite{Makhl Silv Not}. In this section we review the results and study, in terms of cohomology,  the problem of extending a formal deformation of order $k-1$ to a deformation of order $k$. we consider multiplicative Hom-associative algebras and multiplicative Hom-Lie algebras.

Let $\K[[t]]$ be the power series
 ring in one variable $t$ and coefficients in $\K$ and $A[[t]]$ be the set of formal series whose coefficients are elements of the vector space $A$,($A[[t]]$ is obtained
  by extending the coefficients domain of $A$ from $\K$ to $\K[[t]]$), Given a $\K$-bilinear map $\varphi : A \times A\rightarrow A$, it admits naturally
  an extension to a $\K[[t]]$-bilinear map $\varphi : A[[t]]\times A[[t]]\rightarrow A[[t]]$, that is, if $x=\sum_{i\geq0}a_it^i$ and $y=\sum_{j\geq0}y_jt^j$
   then $\varphi(x,y)=\sum_{i\geq0,j\geq0}t^{i+j}\varphi(a_i,b_j)$. The same holds for linear maps.

\subsection{Deformation of Hom-associative algebras}
\begin{defn}
Let $(\A,\mu,\alpha)$    be a  Hom-associative algebra.
A formal  deformation of the  Hom-associative algebra $\A$ is given by  a $\K[[t]]$-bilinear map  $$\mu_t : \A[[t]]\times \A[[t]]\longrightarrow \A[[t]]$$ of the form
$\mu_t=\sum_{i\geq0}t^i \mu_i$ where each $\mu_i$ is a $\K$-bilinear-map  $\mu_i : \A\times \A\rightarrow \A$ (extended to be $\K[[t]]$-bilinear), and $\mu_0=\mu$
 such that for $x, y, z\in \A$ the following condition
 \begin{equation}\label{def ass}
 \mu_t(\mu_t(x,y),\alpha(z))=\mu_t(\alpha(x),\mu_t(y,z))
 \end{equation}

The deformation is said to be of order $k$ if $\mu_t=\sum_{i\geq0}^{k}t^i \mu_i$.
\end{defn}
The identity (\ref{def ass}) is called  deformation equation of the Hom-associative algebra
and  may be written $$\sum_{i\geq0,j\geq0}t^{i+j}\big(\mu_i(\mu_j(x,y),\alpha(z))-\mu_i(\alpha(x),\mu_j(y,z))\big)=0,$$
or $$\sum_{s\geq0}t^{s}\sum_{i\geq0}\big(\mu_i(\mu_{s-i}(x,y),\alpha(z))-\mu_i(\alpha(x),\mu_{s-i}(y,z))\big)=0,$$
which  is equivalent to the following infinite system of equations
$$\sum_{i\geq0}\big(\mu_i(\mu_{s-i}(x,y),\alpha(z))-\mu_i(\alpha(x),\mu_{s-i}(y,z))\big)=0,\;\; \;\hbox{for}\; s=0,1,2,...$$
i.e.
\begin{equation}\label{defo} \sum_{i\geq0}\mu_i\circ_\alpha\mu_{s-i}=0,\; \;\hbox{for}\; s=0,1,2,...\end{equation}

 In particular,\\For $s=0$, we have $ \mu_0\circ_\alpha\mu_0=0$ which corresponds to the Hom-associativity of $\A$.\\
 For $s=1$ we have $\; \mu_0\circ_\alpha\mu_1+\mu_1\circ_\alpha\mu_0=0$ which is equivalent to $\delta_{Hom}^2\mu_1=0$ $(i.e.D(\mu_1)=[\mu,\mu_1]_\alpha^\Delta=0)$.  It turns out that   $\mu_1$ is always a 2-cocycle.

 For $s\geq2$, the identity  \eqref{defo} is equivalent to : $$\delta_{Hom}^2\mu_s=-\sum_{p+q=s}\mu_p\circ_\alpha\mu_q=\frac{1}{2}\sum_{p+q=s,p>0,q>0} [\mu_p,\mu_q]_\alpha^\Delta,$$
 where, $\mu_p\circ_\alpha\mu_q=j_{\mu_q}^\alpha\mu_p $   (see Section \ref{sec3} for the definitions of $j_{\mu_q}^\alpha\mu_p$ and $[.,.]_\alpha^\Delta$).

 \begin{defn}Let $(\A,\mu,\alpha)$ be a Hom-associative algebra.
Given two deformations  $\A_t=(\A,\mu_t,\alpha)$ and  $\A'_{t}=(\A,\mu_t^{'},\alpha)$ of  $\A$ where $\mu_t=\sum_{i\geq0}t^i \mu_i$ and $\mu'_t=\sum_{i\geq0}t^i \mu'_i$
with $\mu_0=\mu$, $\mu'_0=\mu$. We say  that $\A_t$ and $\A'_t$  are equivalents if there exists a formal automorphism $(\phi_t)_{t\geq0} : \A[[t]]\rightarrow \A[[t]]$
that may be written in the form $\phi_t=\sum_{i\geq0}\phi_i t^i$ where $\phi_i\in End(\A)$ and $\phi_0= id$ such that \begin{eqnarray}\label{defo 1}
\phi_t(\mu_t(x,y))&=&\mu_t^{'}(\phi_t(x),\phi_t(y)) \quad \hbox{for} \;\;x, y \in \A[[t]],
 \\ \label{defo 2} \phi(\alpha(x))&=&\alpha(\phi(x))\end{eqnarray}
A deformation $\A_t$ of $\A$ is said to be trivial if and only if $\A_t$ is equivalent to $\A$ (viewed as an algebra over $\A[[t]]$.)
\end{defn}
The identity (\ref{defo 1}) may be written :
for all $x, y \in \A $
$$\sum_{i\geq0,j\geq0,}t^{i+j}\big(\phi_i(\mu_j(x, y))-\sum_{i\geq0,j\geq0,k\geq0}t^{i+j+k}\mu_j(\phi_i(x), \phi_k(y))\big)=0.$$
i.e.
$$\sum_{i\geq0,s\geq0}t^{s}\big(\phi_i(\mu_{s-i}(x, y))\big)-\sum_{i\geq0,j\geq0,s\geq0}t^{s}\big(\mu_j(\phi_i(x), \phi_{s-i-j}(y))\big)=0.$$
Then
$$\sum_{i\geq0}\big(\phi_i(\mu_{s-i}(x, y))-\sum_{j\geq0}\mu_j(\phi_i(x), \phi_{s-i-j}(y))\big)=0 \;\hbox{for}\; s=0,1,2,... $$
In particular, for $s=0$ we have $\mu_0=\mu'_0,$
and  for  $s=1$ $$\phi_0(\mu_1(x,y))+\phi_1(\mu_0(x, y))=\mu'_0(\phi_0(x), \phi_{1}(y))+\mu'_0(\phi_1(x), \phi_0(y))\mu'_1(\phi_0(x), \phi_{0}(y)).$$
Since   $\phi_0=id$    then
\begin{equation}\label{integrability1}\mu'_1(x, y)=\mu_1(x, y)+\phi_1(\mu_0(x, y))-\mu'_0(x, \phi_{1}(y))-\mu'_0(\phi_1(x), y).
\end{equation}
Therefore two  2-cocycles corresponding  to two equivalent deformations
are cohomologous.

\begin{defn}Let $(\A,\mu,\alpha)$ be a Hom-associative algebra, and $\mu_1$ be an element of $Z_{Hom}^2(\A,\A)$,
 the 2-cocycle $\mu_1$ is said integrable if there exists a family  $(\mu_t)_{t\geq0}$ such that $\mu_t=\sum_{i\geq0}t^i\mu_i$ defines a formal deformation $\A_t=(\A[[t]],\mu_t,\alpha)$ of $\A$.
\end{defn}
According to identity \ref{integrability1}, the integrability of $\mu_{1}$ depends only
on its cohomology class.
Thus, we get the following:
\begin{thm}
Let $(\A,\mu,\alpha)$ be a Hom-associative algebra and $\A_t=(\A[[t]],\mu_t,\alpha)$ be a one-parameter formal deformation of $\A$, where $\mu_t=\sum_{i\geq0}t^i\mu_i$. Then there exists an equivalent deformation $\A'_t=(\A[[t]],\mu'_t,\alpha)$, where $\mu'_t=\sum_{i\geq0}t^i\mu'_i$ such that $\mu'_1\in Z_{Hom}^2(\A,\A)$ and $\mu'_1$ does not belong to $ B_{Hom}^2(\A,\A).$

Hence, If  $H_{Hom}^2(\A,\A)= 0$ then every formal deformation  is equivalent to a trivial deformation.
\end{thm}
 Hom-associative algebras for which every formal deformation is equivalent to a trivial deformation are said to be analytically rigid. The nullity of the second cohomology group  ($H_{Hom}^2(\A,\A)= 0$) gives a sufficient criterion for rigidity.

In the following we assume that  $H_{Hom}^2(\A,\A)\neq0$, then one may obtain nontrivial one-parameter formal deformations. We consider the problem of extending a one parameter formal deformation of order $k-1$ to a deformation of order $k$.
\begin{thm}
Let $(\A,\mu,\alpha)$ be a Hom-associative algebra and $\A_t=(\A[[t]],\mu_t,\alpha)$ be an order $k-1$ one-parameter formal deformation of $\A$, where $\mu_t=\sum_{i\geq0}^{k-1}t^i\mu_i$.

Then  $\psi(\mu_1,...,\mu_{k-1})=\frac{1}{2}\sum_{p+q=k-1,p>0,q>0} [\mu_p,\mu_q]_\alpha^\Delta\in Z_{Hom}^3(\A,\A)$ ( i.e. $\psi\in Z_D^2(\A,\A)$).

 Therefore the deformation  extends to a deformation of order $k$ if and only if   $\psi(\mu_1,...,\mu_{k})$ is a coboundary.
\end{thm}
\begin{proof}
We start by defining the linear map $\smile :C(\A,\A) \times C(\A,\A)\rightarrow C(\A,\A)$ by   $$\varphi\smile\psi(x_0,...,x_{a+b})=\mu_0(\varphi(x_0,...,x_a),\psi(x_{a+1},...,x_{a+b+1})),$$ for $\varphi\in C^a(\A,\A),\psi\in C^b(\A,\A)\;\;\hbox{and for}\; x_0,...x_{a+b+1}\in \A.$
Then, $$ \delta_{Hom}^3(\mu_p\circ_\alpha\mu_q)=\delta_{Hom}^2\mu_p\circ_\alpha\mu_q-\mu_p\circ_\alpha\delta_{Hom}^2\mu_q-\mu_p\smile \mu_q+\mu_q\smile\mu_p$$
Notice that $$\sum_{p+q=k, p>0, q>0}\mu_q\smile \mu_p-\sum_{p+q=k, p>0, q>0}\mu_p\smile\mu_q=0$$
We have
\begin{align*}
\delta_{Hom}^3\big(\psi(\mu_1,...,\mu_{k})\big)&=\sum_{p+q=k, p>0, q>0}\big(\delta_{Hom}^2\mu_p\circ_\alpha\mu_q- \mu_p\circ_\alpha\delta_{Hom}^2\mu_q\big)
\\ \ &=\sum_{s+l+q=k, q>0, s>0, l>0} (\mu_s\circ_\alpha\mu_l)\circ_\alpha\mu_q-\sum_{s+l+p=k,p>0, s>0, l>0}\mu_p\circ_\alpha(\mu_l\circ_\alpha\mu_r)\\ \ &
=\sum_{s+l+r=k, r>0, s>0, l>0} (\mu_s\circ_\alpha\mu_l)\circ_\alpha\mu_r-\sum_{s+l+r=k,l>0, s>0, r>0}\mu_s\circ_\alpha(\mu_l\circ_\alpha\mu_r)
\end{align*}
Yet, for any $\beta, \varphi, \gamma \in C^1(\A,\A)$ $$(\beta\circ_\alpha\varphi)\circ_\alpha\gamma-\beta\circ_\alpha(\varphi\circ_\alpha\gamma)=
-(\beta\circ_\alpha\gamma)\circ_\alpha\varphi+\beta\circ_\alpha(\gamma\circ_\alpha\varphi)$$

Indeed, let be $x,y,z,t \in \A$
\begin{align*} (\beta\circ_\alpha\varphi)\circ_\alpha\gamma(x,y,z,t)-\beta\circ_\alpha(\varphi\circ_\alpha\gamma)(x,y,z,t) &= \beta(\gamma(\varphi(x,y),\alpha(z)),\alpha^2(t))-\beta(\gamma(\alpha(x),\varphi(y,z)),\alpha^2(t))\\
\ & +\beta(\alpha^2(x),\gamma(\varphi(y,z),\alpha(t)))-\beta(\alpha^2(x),\gamma(\alpha(t),\varphi(z,t)))\\
\ & -\beta(\gamma(\varphi(x,y),\alpha(z)),\alpha^2(t))+ \beta(\alpha(\varphi(x,y)),\gamma(\alpha(z),\alpha(t))\\
\ & +\beta(\gamma(\alpha(x),\varphi(y,z))\alpha^2(t))-\beta(\alpha^2(x),\gamma(\varphi(y,z),\alpha(t)))\\
\ & - \beta(\gamma(\alpha(x),\alpha(y)),\alpha(\varphi(z,t)))+\beta(\alpha^2(x),\gamma(\alpha(t),\varphi(z,t)))\\
\ &  =\beta(\alpha(\varphi(x,y)),\gamma(\alpha(z),\alpha(t)) - \beta(\gamma(\alpha(x),\alpha(y)),\alpha(\varphi(z,t))).
\end{align*}
Since $$\alpha(\gamma(x,y))=\gamma(\alpha(x),\alpha(y)), \;\; \alpha(\varphi(x,y))=\varphi(\alpha(x),\alpha(y)),$$
then
$$(\beta\circ_\alpha\varphi)\circ_\alpha\gamma(x,y,z,t)-\beta\circ_\alpha(\varphi\circ_\alpha\gamma)(x,y,z,t)=-(\beta\circ_\alpha\gamma)
\circ_\alpha\varphi(x,y,z,t)+\beta\circ_\alpha(\gamma\circ_\alpha\varphi)(x,y,z,t).$$
Thus, $$\delta_{Hom}^3\psi(\mu_1,...,\mu_{k})=0.$$
In the deformation equation corresponding to $\mu_t=\sum_{i\geq0}^{k}t^i\mu_i$ one has moreover the equation
$$\delta_{Hom}^2\mu_k=\psi(\mu_1,...,\mu_{k-1}).
$$
Hence, the $(k-1)$-order formal deformation extends to a $k$-order formal deformation
whenever $\psi$ is a coboundary. \end{proof}
\begin{cor}
If $H_{Hom}^3(\A,\A)=H_D^2(\A,\A)=0$, then any infinitesimal deformation can be extended to a formal deformation.
\end{cor}
The connection to Hom-Poisson algebra has been shown in \cite{Makhl Silv Not}.

\begin{thm}[\cite{Makhl Silv Not}]
Let  $(\A_0,\mu_0,\alpha_0)$ be a commutative
Hom-associative algebra and
${\A_t}=(\A_0[[t]],\mu_t,\alpha_t)$ be a deformation of
$\A_0$. Consider the bracket  defined for $x,y\in
\A$ by $\{x ,y \}=\mu_1 (x,y)-\mu_1 (y,x)$ where
$\mu_1$ is the first order element of the
deformation $\mu_t$. Then $(\A,\mu_0,\{\cdot,\cdot\},\alpha_0)$ is a Hom-Poisson algebra.
\end{thm}
\subsection{Deformation of Hom-Lie algebras}

\begin{defn}
Let  $(\LX,[.,.],\alpha)$  be a Hom-Lie algebra. A one-parameter formal Hom-Lie deformation of $\LX$ is given by the $\K[[t]]$-bilinear map \\$[.,.]_t : \LX[[t]] \times \LX[[t]]\rightarrow \LX[[t]]$ of the form
$$[.,.]_t=\sum_{i\geq0}t^i [.,.]_i$$ where each $[.,.]_i$ is a bilinear map $ [.,.]_i : \LX\times \LX\rightarrow \LX $  (extended to be $\K[[t]]$-bilinear),  $[.,.]=[.,.]_0$  and satisfying the following conditions $$[x,y]_t=-[y,x]_t \;\quad \hbox{skew-symmetry},$$  \begin{equation}\label{Jaco}\circlearrowleft_{x,y,z}\big[\alpha(x),[y,z]_t\big]_t=0  \qquad \hbox{Hom-Jacobi identity}\end{equation}
The deformation is said to be of order $k$ if $[.,.]_t=\sum_{i\geq0}^k t^i [.,.]_i.$
\end{defn}
\begin{rem}
the skew-symmetry of $[.,.]_t$ is equivalent to the skew-symmetry of all $[.,.]_i$ for $i\geq0$.
\end{rem}
The identity (\ref{Jaco}) is called deformation equation of the Hom-Lie algebra and it is equivalent to  $$\circlearrowleft_{x,y,z}\sum_{i\geq0,j\geq0}t^{i+j}\big[\alpha(x),[y,z]_i\big]_j=0$$
i.e.
$$\circlearrowleft_{x,y,z}\sum_{i\geq0,s\geq0}t^{s}\big[\alpha(x),[y,z]_i\big]_{s-i}=0$$
or
$$\sum_{s\geq0}t^s\circlearrowleft_{x,y,z}\sum_{i\geq0}\big[\alpha(x),[y,z]_i\big]_{s-i}=0$$
which is equivalent to the following infinite system \begin{equation}\label{cycl}\circlearrowleft_{x,y,z}\sum_{i\geq0}\big[\alpha(x),[y,z]_i\big]_{s-i}=0, \;\hbox{for}\; s=0,1,2,...\end{equation}
In particular, for $s=0$ we have $ \circlearrowleft_{x,y,z}\big[\alpha(x),[y,z]_0\big]_0$ which  is the Hom-Jacobi identity of $\LX$.
\\The equation, for s=1, leads to $\delta_{HL}^2[.,.]_1=0,$ i.e.  $D[.,.]_1=[[.,.],[.,.]_1]_\alpha^\wedge=0$. Then  $[.,.]_1$ is a 2-cocycle.

For $s\geq2$, the identity (\ref{cycl}) is equivalent to : \begin{align*}
\delta_{HL}^2[.,.]_s(x,y,z)&=-\sum_{p+q=s}\circlearrowleft_{x,y,z}\big[\alpha(x),[y,z]_q\big]_p\\
\ & =\frac{1}{2}\sum_{p+q=s,p>0,q>0} \big[[.,.]_p,[.,.]_q\big]_\alpha^\wedge(x,y,z)
\end{align*}
See Section \ref{secCTilde} for the definition of $[.,.]_\alpha^\wedge$.
\begin{defn}
Let $(\LX,[.,.],\alpha)$ be a Hom-Lie algebra satisfying $[\alpha(x),\alpha(y)]=\alpha([x,y])$. Given
two deformations $\LX_t=(\LX,[.,.]_t,\alpha)$ and  $\LX'_t=(\LX,[.,.]_t^{'},\alpha)$ of $\A$ where $[.,.]_t=\sum_{i\geq0}t^i [.,.]_i$ and $[.,.]'_t=\sum_{i\geq0}t^i [.,.]'_i$  with $[.,.]_0=[.,.]'_0=[.,.]$. We say  that  $\LX_t$ and $\LX'_t$  are equivalents if there exists a formal  automorphism $(\phi_t)_{t\geq0} : \LX[[t]]\rightarrow \LX[[t]]$,  that may be written in the form $\phi_t=\sum_{i\geq0}\phi_i t^i$ where $\phi_i\in End(\LX)$ and $\phi_0= id$, such that $$\phi_t([x,y]_t)=[\phi_t(x),\phi_t(y)]_t^{'}.$$
A deformation $L_t$ is said to be trivial if and only if $\LX_t$ is equivalent to $\LX$ (viewed as an algebra on $\LX[[t]]$.)
\end{defn}
Similarly to Hom-associative algebras, we have that  two  2-cocycles corresponding  to two equivalent deformations
are cohomologous.
\begin{defn}
Let $(\LX,[.,.],\alpha)$ be a Hom-Lie algebra, and $[.,.]_1$ be an element of $Z_{HL}^2(\LX,\LX)$,
 the 2-cocycle $[.,.]_1$ is said to be integrable if there exists a family  $([.,.]_t)_{t\geq0}$ such that $[.,.]_t=\sum_{i\geq0}t^i [.,.]_i$ defines a formal deformation $\LX_t=(\LX,[.,.]_t,\alpha)$ of $\A$.
\end{defn}
One may also prove
\begin{thm}
Let $(\LX,[.,.],\alpha)$ be a Hom-Lie algebra and $\LX_t=(\LX,[.,.]_t,\alpha)$ be a one-parameter formal deformation of $\LX$, where $[.,.]_t=\sum_{i\geq0}t^i [.,.]_i$. Then there exists an equivalent deformation $[.,.]'_t=\sum_{i\geq0}t^i [.,.]'_i$, where $\mu'_t=\sum_{i\geq0}t^i\mu'_i$ such that $[.,.]'_1\in Z_{HL}^2(\LX,\LX)$ and $[.,.]'_1$ does not belong to $ B_{HL}^2(\LX,\LX).$

Hence, If  $H_{HL}^2(\LX,\LX)= 0$ then every formal deformation  is equivalent to a trivial deformation.
\end{thm}

The Hom-Lie algebras whose all formal  deformations  are trivial we said to be rigid. The previous theorem gives a criterion for rigidity.

The obstruction study leads in the case of Hom-Lie algebra to the following theorem.

\begin{thm}
Let $(\LX,[.,.],\alpha)$ be a Hom-Lie algebra and $\LX_t=(\LX,[.,.]_t,\alpha)$ be a $k-1$-order one-parameter formal deformation of $\LX$, where $[.,.]_t=\sum_{i\geq0}^{k-1}t^i [.,.]_i$.

Then  $$\psi([.,.]_1,...,[.,.]_{k-1})=\frac{1}{2}\sum_{p+q=k-1,p>0,q>0} [[.,.]_p,[.,.]_q]_\alpha^\wedge \in Z_{HL}^3(\LX,\LX)$$ i.e $\psi\in \tilde{Z}_D^2(\LX,\LX)$.

  Therefore the deformation  extends to a deformation of order $k$ if and only if   $\psi([.,.]_1,...,[.,.]_{k-1})$ is a coboundary.
\end{thm}
\begin{proof}
with a direct computation we have
$$\delta_{HL}^3\big(\psi([.,.]_1,...,[.,.]_{k})\big)(x,y,z,t)=A_1+B_1+C_1$$
where
\begin{align*}
A_1&=\sum_{p+q=k, p>0, q>0}\big(\delta_{HL}^2[.,.]_q(\alpha(x),\alpha(t),[y,z]_p)+\delta_{HL}^2[.,.]_q(\alpha(y),\alpha(z),[x,t]_p)
\\ \ & +\delta_{HL}^2[.,.]_q(\alpha(x),\alpha(y),[z,t]_p)+\delta_{HL}^2[.,.]_q(\alpha(x),\alpha(z),[t,y]_p)\\ \ & +\delta_{HL}^2[.,.]_q(\alpha(y),\alpha(t),[z,x]_p) +\delta_{HL}^2[.,.]_q(\alpha(z),\alpha(t),[x,y]_p)\big),
\\
B_1&=\sum_{p+q=k, p>0, q>0}\big(\big[\alpha^2(x),\delta_{HL}^2[.,.]_p(z,y,t)\big]_q
+\big[\alpha^2(y),\delta_{HL}^2[.,.]_p(x,z,t)\big]_q\\ \ & +\big[\alpha^2(z),\delta_{HL}^2[.,.]_p(y,x,t)\big]_q+\big[\alpha^2(t),\delta_{HL}^2[.,.]_p(x,y,z)\big]_q\big),
\\ C_1&=\sum_{p+q=k, p>0, q>0}\big(-\big[[\alpha(z),\alpha(t)]_p,[\alpha(x),\alpha(y)]_q\big]_0-\big[[\alpha(t),\alpha(y)]_p,
[\alpha(x),\alpha(z)]_q\big]_0
\\ \ &-\big[[\alpha(y),\alpha(z)]_p,[\alpha(x),\alpha(t)]_q\big]_0
-\big[[\alpha(x),\alpha(t)]_p,[\alpha(y),\alpha(z)]_q\big])\\ \ & -\big[[\alpha(z),\alpha(x)]_p,[\alpha(y),\alpha(t)]_q\big]_0
  -\big[[\alpha(x),\alpha(y)]_p,[\alpha(z),\alpha(t)]_q\big]_0\big)
  \\ \ &=0.
  \end{align*}
 since $$\delta_{HL}^2[.,.]_m=-\sum_{r+s=m}\circlearrowleft_{x,y,z}\big[\alpha(x),[y,z]_r\big]_s.$$

 Then
$$A_1=A_{11}+A_{12},$$ where \begin{align*}
A_{11}&=\sum_{p+s+l=k}\big(\circlearrowleft_{z,y,t}[\alpha^2(x),[\alpha(z),[t,y]_p]_s]_l+
\circlearrowleft_{z,y,t}[\alpha^2(y),[\alpha(x),[t,z]_p]_s]_l\\ \ & +\circlearrowleft_{z,y,t}[\alpha^2(z),[\alpha(t),[x,y]_p]_s]_l+\circlearrowleft_{z,y,t}[\alpha^2(t),
[\alpha(x),[z,y]_p]_s]_l\big),\\
A_{12}&=\sum_{p+s+l=k}\big(\big[[\alpha(y),\alpha(z)]_p,[\alpha(x),\alpha(t)]_s\big]_l+\big[[\alpha(x),
\alpha(t)]_p,[\alpha(y),\alpha(z)]_s\big]_l\\ \ &+\big[[\alpha(z),\alpha(t)]_p,[\alpha(x),\alpha(y)]_s\big]_l
+\big[[\alpha(t),\alpha(y)]_p,[\alpha(x),\alpha(z)]_s\big]_l\\ \ &+\big[[\alpha(z),\alpha(x)]_p,[\alpha(y),\alpha(t)]_s\big]_l
+\big[[\alpha(x),\alpha(y)]_p,[\alpha(z),\alpha(t)]_s\big]_l\big),\\
B_1&=\sum_{q+s+l=k}\big(\circlearrowleft_{z,y,t}[\alpha^2(x),[\alpha(z),[y,t]_l]_s]_q+
\circlearrowleft_{z,y,t}[\alpha^2(y),[\alpha(x),[z,t]_l]_s]_q\\ \ &+\circlearrowleft_{z,y,t}[\alpha^2(z),[\alpha(t),[y,x]_l]_s]_q+\circlearrowleft_{z,y,t}[\alpha^2(t),
[\alpha(x),[y,z]_l]_s]_q\big).
\end{align*}
We have
$$A_{11}+B_1=0\;\; \hbox{and} \;\;\;A_{12}=0 .$$
Therefore $$\delta_{HL}^3\big(\psi([.,.]_1,...,[.,.]_{k})\big)(x,y,z,t)=0.$$
In the deformation equation corresponding to $[.,.]_t=\sum_{i\geq0}^k t^i [.,.]_i$ one has moreover the equation
$$\delta_{HL}^2[.,.]_{k}=psi([.,.]_1,...,[.,.]_{k-1}).
$$
Hence, the $(k-1)$-order formal deformation extends to a $k$-order formal deformation
whenever $\psi$ is a coboundary.
\end{proof}
\begin{cor}
If $H_{HL}^3(\LX,\LX)=\tilde{H}_D^2(\LX,\LX)=0$, then any infinitesimal deformation can be extended to a formal deformation.
\end{cor}
As in the Hom-associative case the space $H_{HL}^2(\LX,\LX)$ classify
the infinitesimal deformation and the space $H_{HL}^3(\LX,\LX)$ contains the
obstructions. Also we recover the results of the classical cases.


\begin{thebibliography}{99}


\bibitem {F.Ammar} Ammar F. and Makhlouf A., \emph{Hom-Lie algebras and Hom-Lie admissible superalgebras}, . arXiv:0906.1668, (2009).
\bibitem{right} Dzhumadil'Daev A., \emph{Cohomology and deformations of right symmetic algebras}, arxiv: math/9807065v1 (1998).


\bibitem{Extensions}Fialowski A. and Penkava M., \emph{Extensions of (super) Lie algebra}, Commun. Contemp. Math.  \textbf{11},  no. 5  (2009), 709--737.




\bibitem{Fuks} Fuks D.B., \emph{cohomology of infinite-dimensional Lie algebras}, Plenum, New York, 1986.

\bibitem{Gerst coho} Gerstenhaber M.,\emph{The cohomology structure of an associative ring},  Ann of Math., \textbf{78} (2) (1963), 267--288.

\bibitem{Gerst def} Gerstenhaber M.,\emph{On the deformation of rings and algebras}, Ann of Math.(2)\textbf{79} (1) (1964), 59--108.

\bibitem{Gohr} Gohr A., \emph{On  Hom-algebras with surjective twisting},arXiv:0904.4874v2, (2009).

 \bibitem{Harwig Silv} Harwig J.T., Larsson D. and Silvestrov S.D., \emph{Deformations of Lie algebras using $\sigma-$derivations}, J.Algebra \textbf{295} (2006).

\bibitem{semi s}Jin Q., Li X.,\emph{Hom-Lie algebra structures on semi-simple Lie algebras}, Journal of Algebra \textbf{319} (2008).


\bibitem{Lecomte} Lecomte P.A.B. , Schicketanz H., \emph{The multigraded Nijenhuis-Richardson algebra, its universal property and applications}, arxiv: math/920T257 (1992).
 \bibitem{Makhl Silv Hom}   Makhlouf  A., Silvestov S., \emph{Hom-algebra structures}, J. Gen. Lie Theory Appl. Vol \textbf{2} (2), pp 51-64 (2008).

     \bibitem{HomHopf} Makhlouf A., Silvestrov S. D., \emph{Hom-Lie
admissible Hom-coalgebras and Hom-Hopf algebras}, In "Generalized Lie theory in Mathematics, Physics and Beyond. S. Silvestrov, E. Paal, V. Abramov, A. Stolin, Editors". Springer-Verlag, Berlin, Heidelberg, Chapter 17, pp 189-206, (2009).

\bibitem{HomAlgHomCoalg} Makhlouf A., Silvestrov S. D.,
\emph{Hom-Algebras and Hom-Coalgebras}, Preprints in Mathematical Sciences, Lund University, Centre for Mathematical Sciences, Centrum Scientiarum Mathematicarum, (2008:19) LUTFMA-5103-2008 and in arXiv:0811.0400 [math.RA] (2008). To appear in Journal of Algebra and its Applications.

\bibitem{Makhl Silv Not} Makhlouf A., Silvestrov S. D.,
\emph{Notes on Formal deformations of Hom-Associative and Hom-Lie algebras}, Preprints in Mathematical Sciences, Lund University, Centre for Mathematical Sciences, Centrum Scientiarum Mathematicarum, (2007:31) LUTFMA-5095-2007. arXiv:0712.3130v1 [math.RA] (2007). To appear in Forum Mathematicum.

\bibitem{NijenhuisRichardson} Nijenhuis A. and Richardson R., \emph{Deformation of Lie algebras structures}, Jour. Math. Mech. \textbf{17} (1967), 89--105.

 \bibitem{Rotki} Rotkiewicz  M., \emph{Cohomology  Ring of n-Lie Algebras}, Extracta Mathematicae Vol. \textbf{20}, Num. 3(2005), 219--232.

     \bibitem{stasheff} Stasheff J.D., \emph{The intrinsic bracket on the deformation complex of an associative algebra}, Journal of Pure and Applied Algebra \textbf{89} (1993), 231--235.

\bibitem{Yau:EnvLieAlg} Yau D.:
{\emph Enveloping algebra of Hom-Lie algebras,}  J. Gen.
Lie Theory Appl. \textbf{2} (2), 95--108 (2008).
  \bibitem{Yau homol} Yau D., \emph{Hom-algebras and homology},  J. Lie Theory 19 (2009) 409-421.


 \bibitem{Yau Module} Yau D., \emph{Module  Hom-algebras}, arXiv:0812.4695v1 (2008).



\end{thebibliography}
\end{document}